\pdfoutput=1
\documentclass{article}
\usepackage{graphicx,hyperref}
\usepackage{amsmath,amsfonts,amssymb,latexsym,amsthm}
\usepackage{multicol,multirow}
\usepackage{placeins}

\usepackage{bm}
\providecommand{\B}{}
\renewcommand{\B}{\bm}
\newcommand{\D}{\partial}
\renewcommand{\le}{\leqslant}
\renewcommand{\ge}{\geqslant}
\newtheorem{theorem}{Theorem}
\newtheorem{lemma}[theorem]{Lemma}
\newtheorem{proposition}[theorem]{Proposition}
\newtheorem{corollary}[theorem]{Corollary}

\newtheorem{remark}{Remark}
\theoremstyle{definition}

\def\newhbar{\raise0.7ex \hbox{--}\mkern-9mu \hbox{$h$}}

\begin{document}

\title{Convergence rate of the spectral difference method on regular triangular meshes}

\author{P.A.~Bakhvalov}

\date{April 16, 2024}

\maketitle

\begin{abstract}
We consider the spectral difference method based on the $p$-th order Raviart~-- Thomas space ($p = 1, 2, 3$) on regular triangular meshes for the scalar transport equation. The solution converges with the order $p$ if the transport velocity is parallel to a family of mesh edges and with the order $p+1$ otherwise. We prove this fact for \mbox{$p=1$} and show it for $p=1, 2, 3$ in numerical experiments.
\end{abstract}

\medskip

\sloppy

\section{Introduction}

The spectral difference (SD) method is a high order method for solving hyperbolic problems on unstructured meshes. Like in the discontinuous Galerkin (DG) method, a mesh function is a discontinuous piecewise $p$-th order polynomial function, and a Riemann solver is used to calculate the numerical fluxes. The SD method attracts attention because of its simplicity compared to DG. Its main drawback is that SD is not based on a solid mathematical background. In particular, there is no $L_2$-stability proof on unstructured meshes.

The spectral difference method was proposed for unstructured triangular meshes in \cite{Liu2006, SD1} and for quadrilateral meshes in \cite{SD2,SD5,SD3,SD4}. The stability in the 1D case was proved in \cite{SD6}. However, it was found that for $p \ge 2$ this scheme is unstable even on regular-triangular meshes \cite{Abelee2008}. This issue was overcome by the spectral difference method based on the Raviart~-- Thomas space (SD-RT), which was proposed for $p \le 3$ in \cite{SD8}. It was generalized to $p \le 6$ in \cite{Veilleux2021}, to tetrahedral meshes in \cite{Veilleux2022}, and to mixed-element 3D meshes in \cite{SaezMischlich2023}.

In this paper we study the accuracy of the SD-RT method for the Cauchy problem for the transport equation
\begin{equation}
\begin{gathered}
\frac{\D v}{\D t} + \B{\omega} \cdot \nabla v = 0, \quad \B{r} \in \mathbb{R}^2, \quad 0 < t < t_{\max},
\\
v(0,\B{r}) = v_0(\B{r}).
\label{eq_TE}
\end{gathered}
\end{equation}
The transport velocity \mbox{$\B{\omega}=  (\omega_x, \omega_y)^T$} is constant, and the initial data $v_0$ is sufficiently smooth and periodic with the periodic cell $(0,1)^2$.

The convergence rate $P$ in a quadratic norm of a stable $p$-exact numerical method for \eqref{eq_TE} usually belongs to the range $[p, p+1]$. For finite-difference schemes, there holds $P = p$. On 1D non-uniform meshes, for both polynomial-based finite-volume schemes and the discontinous Galerkin method there holds $P = p+1$.

On unstructured meshes, the situation becomes more complicated. The DG($p$) method converges with the order $p+1/2$ under the minimal angle condition \cite{Johnson1986}. The $(p+1)$-th order convergence is known if the angle between the mesh edges and the transport velocity is bounded from below \cite{Richter1988}. The importance of this assumption was demonstrated by Peterson \cite{Peterson1991} who constructed a sequence of meshes where DG($p$) converges with the order $p+1/2$.

A similar effect holds for finite-volume schemes. On unstructured meshes, one usually observes a convergence rate close to $p+1$ \cite{Diskin2011}. But a lower convergence rate may happen. For instance, the convergence rates $3/2$ and $5/4$ were demostrated in a Peterson-type counterexample for 1-exact edge-based schemes  \cite{Bakhvalov2023b}.

Now we are concerned with the case of regular (also called translationally invariant, TI) triangular meshes, i. e. meshes that are invariant with respect to the translation by each edge. A triangular TI-mesh is the image of an infinite mesh of regular triangles mesh under a linear map.  Two adjacent triangles form a periodic cell of the mesh, so each scheme on a triangular TI-mesh can be interpreted as a scheme with several degrees of freedom per cell on a uniform structured mesh. For a given direction of the transport velocity, the convergence rate of such schemes may be either $P=p$ or $P=p+1$, see \cite{Bakhvalov2023}. The $(p+1)$-th order convergence of DG($p$) in this case is a corollary of the $(p+1/2)$-th order convergence. Finite-volume schemes also exhibit the $(p+1)$-th order convergence provided that the dissipation is good enough so the checkerboard function shown at Fig.~\ref{fig:checkerboard} is not a steady numerical solution.

In this paper we consider the SD-RT($p$) method for the transport equation on translationally invariant triangular meshes, $p = 1, 2, 3$. We show that the numerical solution converges with the order $p$ if the transport velocity is parallel to a family of mesh edges and with the order $p+1$ otherwise. We prove this effect for $p = 1$ and demonstrate it numerically for $p = 1, 2, 3$.

\begin{figure}[t]
\centering
\includegraphics[width=0.3\textwidth]{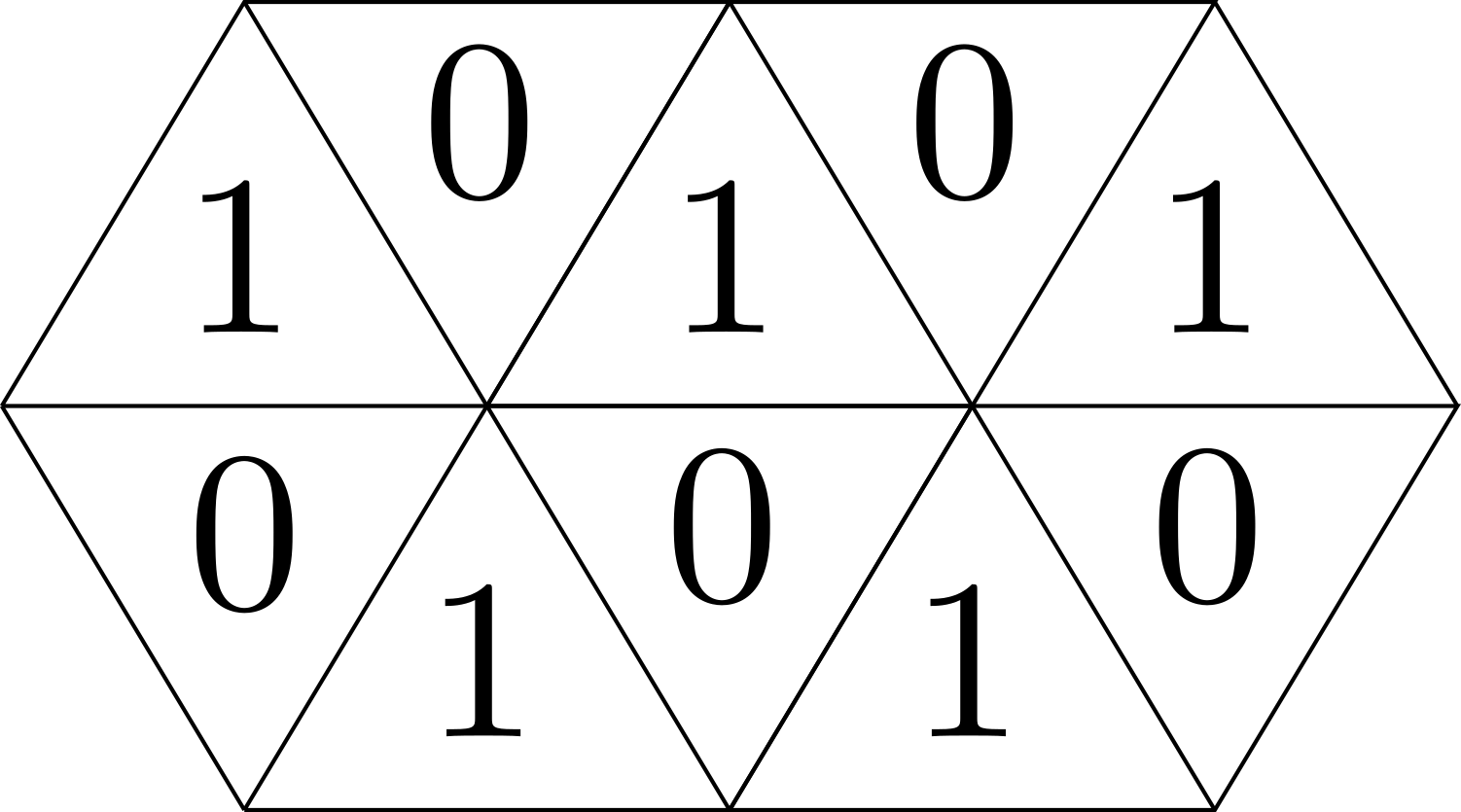}
\caption{The checkerboard function on a regular triangular mesh}\label{fig:checkerboard}
\end{figure}

\section{The SD-RT method}


\subsection{The Raviart~-- Thomas space}
\label{sect:RT}

Let $P_p$, $p \in \mathbb{N} \cup \{0\}$, be the space of $p$-th order polynomials of two variables. The Raviart~-- Thomas space of order $p$ is formally defined as 
$$
RT_p = P_p^2 + (x,y)^T P_p.
$$
In particular,
$$
RT_1 = span\left\{
\left(\begin{array}{c}1 \\ 0\end{array}\right),
\left(\begin{array}{c}x \\ 0\end{array}\right),
\left(\begin{array}{c}y \\ 0\end{array}\right),
\left(\begin{array}{c}0 \\ 1\end{array}\right),
\left(\begin{array}{c}0 \\ x\end{array}\right),
\left(\begin{array}{c}0 \\ y\end{array}\right),
\left(\begin{array}{c}x^2 \\ xy\end{array}\right),
\left(\begin{array}{c}xy \\ y^2\end{array}\right)
\right\}.
$$
Since $dim\, P_p = (p+1)(p+2)/2$, and the dimension of the space of homogeneous $p$-order polynomials is $p+1$, then $dim\,RT_p = (p+1)(p+3)$. The divergence operator maps $RT_p$ onto $P_p$. For each $\B{n} \in \mathbb{R}^2$ and $\B{f} \in RT_p$, the function $\B{f} \cdot \B{n}$ is a $p$-th order polynomial on each line orthogonal to $\B{n}$. For more information about the Raviart~-- Thomas space see \cite{Sayas2013}.

\subsection{SD-RT on a general triangular mesh}
\label{sect:22}

Consider a triangular mesh in $\mathbb{R}^2$, periodic with the periodic cell $(0,1)^2$. Let $\mathcal{E}$ be the set of mesh triangles. A {\it mesh function} is a generally discontinuous piecewise-polynomial function, namely, $u = \{u_e \in P_p, e \in \mathcal{E}\}$, periodic with the periodic cell $(0,1)^2$.

For each mesh edge, put $p+1$ points, for instance, at the knots of the Gauss~-- Legendre quadrature rule. On each triangle, define $p(p+1)/2$ {\it interior points}. Together, they form the set of {\it flux collocation points}. For $p=1$, on each $e \in \mathcal{E}$, there is one interior point at the mass center. For $p > 1$, the location of interior points is chosen to enforce stability of the resulting scheme \cite{SD8,Veilleux2021}, but there is no algorithm for a general $p$. 

By SD-RT($p$) we denote the spectral difference method based on the Raviart~-- Thomas space of order $p \in \mathbb{N} \cup \{0\}$. For the transport equation \eqref{eq_TE} it has the form
\begin{equation}
\frac{\D u_e}{\D t} + \mathrm{div}\,\B{f}_e[u] = 0,\quad e \in \mathcal{E}, \quad u_e(t) \in P_p,
\label{eq_02}
\end{equation}
where $\B{f}_e[u] \in RT_p$ is defined by the following conditions:
\begin{itemize}
\item for each interior collocation point $\B{r}_j$, $j = 1, \ldots, p(p+1)/2$, inside the triangle $e$ there holds
\begin{equation}
(\B{f}_e[u])(\B{r}_j) = \B{\omega} u(\B{r}_j);
\label{eq_SD}
\end{equation}
\item for each flux collocation point $\B{r}_j$ on $\D e$ there holds
\begin{equation}
(\B{f}_e[u])(\B{r}_j) \cdot \B{n} = F_n(\B{r}_j)
\label{eq_SD2}
\end{equation}
\end{itemize}
where $F_n(\B{r}_j)$ is defined by solving the Riemann problem:
$$
F_n(\B{r}_j) = (\B{\omega} \cdot \B{n}) \left\{\begin{array}{ll}
u(\B{r}_j + 0\B{n}), & \B{\omega} \cdot \B{n} < 0; \\
u(\B{r}_j - 0\B{n}), & \B{\omega} \cdot \B{n} \ge 0.
\end{array}\right.
$$
Here $\B{n}$ is the unit normal vector to the edge and $u(\B{r}_j \pm 0\B{n})$ are the limit values from both sides of the edge.

To define $\B{f}_e[u] \in RT_p$, we have $p(p+1)/2$ vector equations and $3(p+1)$ scalar equations, ${dim}\,RT_p$ in total. Thus, the number of equations is equal to the number of unknowns. The non-degeneracy of this system is provided by the location of the interior collocation points.

Unless specifically stated, the initial data for the semidiscrete problem are given by 
\begin{equation}
u(0) = \Pi_h v_0
\label{eq_id}
\end{equation}
where $\Pi_h$ is the Lagrangian mapping. Our results also hold if the orthogonal (in the sense of $L_2$) projection is used, see Remark~\ref{remark:2}.

By construction, SD-RT($p$) is $p$-exact in the sense of $\Pi_h$, i. e. for each $f \in P_p$ the function $u(t) = \Pi_h(g(t,\ \cdot\ ))$, $g(t, \B{r}) = f(\B{r} -\B{\omega}t)$, satisfies \eqref{eq_02}.


Note that the replacement of \eqref{eq_SD} by
$$
\int\limits_{e} (\B{f}_e[u])(\B{r}) \phi(\B{r}) d\B{r} = \int\limits_{e} \B{\omega} u(\B{r}) \phi(\B{r}) d\B{r}, \quad \forall \phi \in P_{p-1},
$$
yields the discontinuous Galerkin method. For $p=0$, there are no interior collocation points, so SD-RT(0) coincides with DG(0) and with the basic finite-volume method.

\subsection{SD-RT(1) on a right-triangular mesh}
\label{sect:regmesh}

Now we specify SD-RT(1) for a regular triangular mesh. Since an affine transform of the mesh together with the transport velocity keeps the scheme unmodified, we restrict our analysis to the meshes of right isosceles triangles.

Let the mesh nodes have radius-vectors $(mh, nh)^T$, $m, n \in \mathbb{Z}$, $1/h \in \mathbb{N}$, and the mesh edges connect each node $(mh, nh)^T$ to
$$
((m+1)h, nh)^T, \quad (mh, (n+1)h)^T, \quad ((m+1)h, (n-1)h)^T.
$$
The scheme coefficients are piecewise linear functions of $\B{\omega}$, and their gradients are discontinuous when $\B{\omega}$ is parallel to a mesh edge. Without loss, we consider the case $\omega_x, \omega_y \ge 0$, $\omega_x + \omega_y > 0$. 

\begin{figure}[t]
\centering
\includegraphics[width=0.2\textwidth]{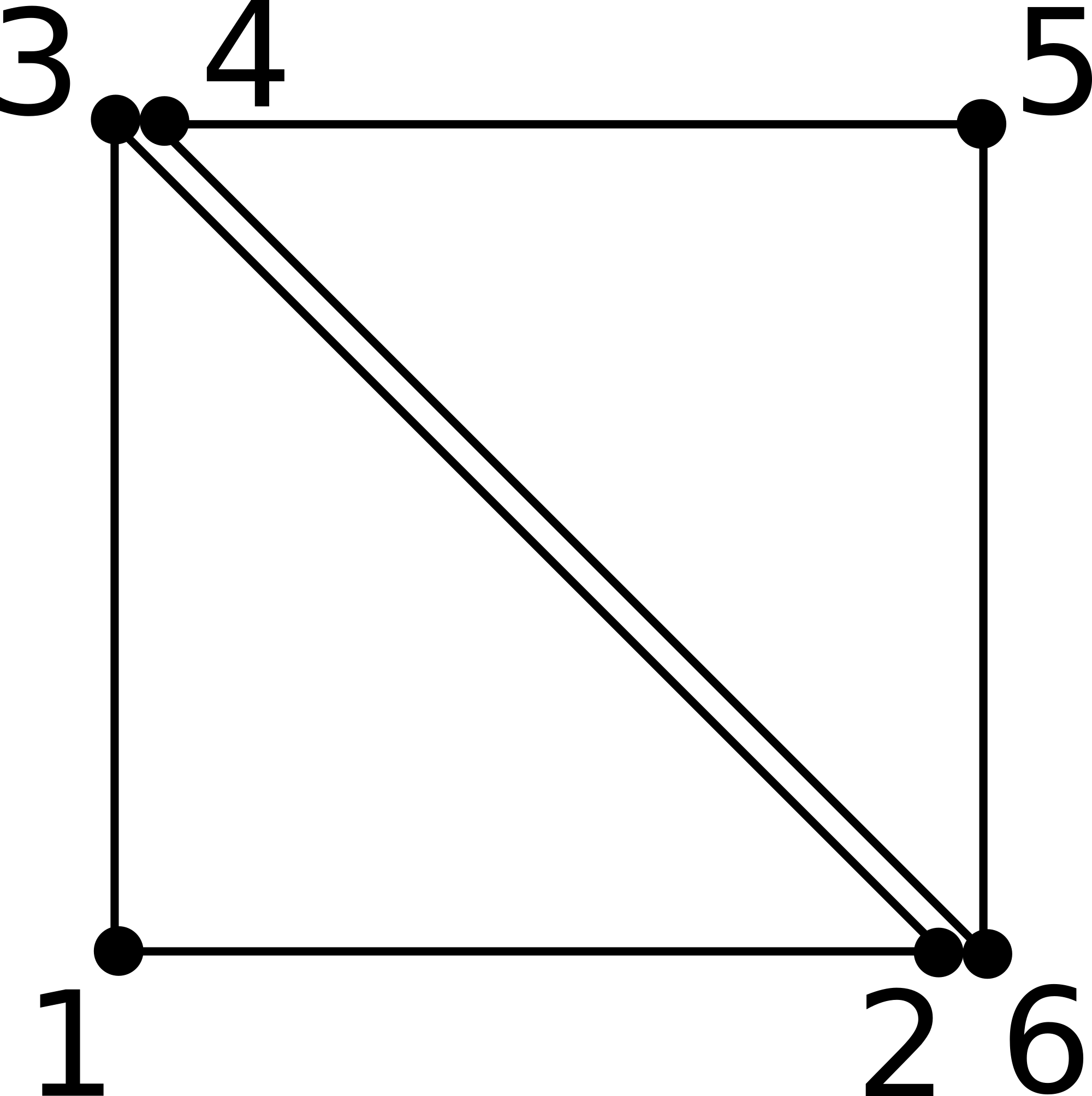}
\caption{Solution points}\label{fig:sol_points}
\end{figure}

Represent a mesh function on each triangle by its values at the vertices. Let  $\B{\eta} = (\eta_x, \eta_y) \in \mathbb{Z}^2$. The point values at the vertices on the periodic cell of the mesh
$$
[\eta_x h, (\eta_x+1)h] \times [\eta_y h, (\eta_y+1)h]
$$
are numerated as shown in Fig.~\ref{fig:sol_points}. Denote them by \mbox{$u_{\B{\eta}} \in \mathbb{C}^{6}$}. Then $\{u_{\B{\eta}}, \B{\eta} \in \mathbb{Z}^2\}$ represents a mesh function.

In this notation, SD-RT(1) on the right-triangular mesh with step $h$ takes the form
\begin{equation}
\frac{du_{\B{\eta}}}{dt}(t) +
\frac{1}{h}\sum\limits_{\B{\zeta} \in \mathcal{S}} L_{\B{\zeta}} u_{\B{\eta}+\B{\zeta}}(t) = 0, \quad
\B{\eta} \in \mathbb{Z}^2,
\label{eq3}
\end{equation}
where $\mathcal{S} = \{(0,0), (-1,0), (0,-1)\}$ is the scheme stencil,
$$
L_{\B{\zeta}} = \omega_x L_{\B{\zeta}}^x + \omega_y L_{\B{\zeta}}^y,
$$
$L_{\B{\zeta}}^x$ and $L_{\B{\zeta}}^y$ are real-valued $(6 \times 6)$-matrices.
The matrices $L_{\B{\zeta}}^x$ have the form $L_{(0,-1)}^x = 0$,
$$
L_{(0,0)}^x = \left(\begin{array}{cccccc}
3 & 1 & 1 & 0 & 0 & 0 \\
-3 & 1 & -2 & 0 & 0 & 0 \\
0 & 1 & 4 & 0 & 0 & 0 \\
0 & -1 & -4 & 3 & 1 & 1\\
0 & 2 & 2 & -3 & 1 & -2 \\
0 & -4 & -1 & 0 & 1 & 4
\end{array} \right), \quad
L_{(-1,0)}^x = \left(\begin{array}{cccccc}
0 & 0 & 0 & 0 & -1 & -4 \\
0 & 0 & 0 & 0 & 2 & 2 \\
0 & 0 & 0 & 0 & -4 & -1 \\
0 & 0 & 0 & 0 & 0 & 0 \\
0 & 0 & 0 & 0 & 0 & 0 \\
0 & 0 & 0 & 0 & 0 & 0
\end{array} \right).
$$
The matrix $L_{(\zeta_x, \zeta_y)}^y$ results from $L_{(\zeta_y, \zeta_x)}^x$ by the permutation of rows 2 and 3, columns 2 and 3, rows 4 and 6, colunms 4 and 6: $L_{(-1,0)}^y = 0$,
$$
L_{(0,0)}^y = \left(\begin{array}{cccccc}
3 & 1 & 1 & 0 & 0 & 0 \\
0 & 4 & 1 & 0 & 0 & 0 \\
-3 & -2 & 1 & 0 & 0 & 0 \\
0 & -1 & -4 & 4 & 1 & 0 \\
0 & 2 & 2  & -2 & 1 & -3 \\
0 & -4 & -1  & 1 & 1 & 3
\end{array} \right), \quad
L_{(0,-1)}^y = \left(\begin{array}{cccccc}
0 & 0 & 0 & -4 & -1 & 0 \\
0 & 0 & 0 & -1 & -4 & 0\\
0 & 0 & 0 & 2 & 2 & 0 \\
0 & 0 & 0 & 0 & 0 & 0 \\
0 & 0 & 0 & 0 & 0 & 0 \\
0 & 0 & 0 & 0 & 0 & 0
\end{array} \right).
$$
Taking the linear interpolation from \eqref{eq3} within each triangle returns us to the form  \eqref{eq_02}.

The Lagrangian mapping $\Pi_h$ takes each $v \in C(\mathbb{R}^2)$ to the function $\Pi_h v$ with the components
$$
(\Pi_h v)_{\B{\eta}, \xi} = v(\B{r}_{\B{\eta}, \xi}),\quad \B{\eta} \in \mathbb{Z}^2, \quad \xi = 1, \ldots, 6,
$$
where $\B{r}_{\B{\eta}, \xi} = (\B{\eta} + \B{a}_{\xi}) h$ is the radius-vector of solution point $\xi$ in block $\B{\eta}$. The vectors $\B{a}_{\xi}$ are based on Fig.~\ref{fig:sol_points}: $\B{a}_1 = (0,0)^T$, $\B{a}_2 = \B{a}_6 = (1,0)^T$, $\B{a}_3 = \B{a}_4 = (0,1)^T$, $\B{a}_5 = (1,1)^T$.

\section{Criterion of the $(p+1)$-th order convergence}

Schemes of the general form \eqref{eq3}, \eqref{eq_id} were studied in \cite{Bakhvalov2023}. We use the notation and some results from that paper. Throughout this section, $\B{\omega}$ is fixed.

Let $\B{m} = (m_x, m_y)$ be a multiindex: $m_x, m_y \in \mathbb{N} \cup \{0\}$. Denote $\vert \B{m} \vert = m_x + m_y$, $\B{r}^{\B{m}} = x^{m_x} y^{m_y}$, $\B{m}! = m_x!\,m_y!$. For $\vert \B{m} \vert = p+1$ denote
\begin{equation}
f^{\B{m}} = - \left(\Pi_1 (\B{\omega} \cdot \nabla) \frac{\B{r}^{\B{m}}}{\B{m}!}\right)_{0} + \sum\limits_{\B{\zeta} \in \mathcal{S}} L_{\B{\zeta}} \left(\Pi_1 \frac{\B{r}^{\B{m}}}{\B{m}!}\right)_{\B{\zeta}}.
\label{eq_trerr}
\end{equation}
Here $\Pi_1$ means $\Pi_h$ with the substitution $h=1$. The vector $f^{\B{m}} \in \mathbb{R}^{6}$ is the coefficient at $\B{m}$-th derivative in the truncation error of the scheme.

Equip the space of mesh functions with the norm
$$
\|f\|^2 = \sum\limits_{\B{\eta} \in \{1, \ldots, 1/h\}^2} h^2 \|f_{\B{\eta}}\|^2,
$$
and use the Euclidean norm on $\mathbb{C}^6$. For $\B{\phi} \in \mathbb{R}^2$, denote
$$
L(\B{\phi}) = \sum\limits_{\B{\zeta} \in \mathcal{S}} \exp(i\phi_x \zeta_x + i\phi_y \zeta_y) L_{\B{\zeta}} .
$$
The scheme is stable with constant $K \in [1, \infty)$ if for each $h$ each $(1/h)$-periodic solution of \eqref{eq3} satisfies $\|u(t)\| \le K \|u(0)\|$. This holds iff
\begin{equation}
\sup_{\B{\phi} \in \mathbb{R}^2} \sup\limits_{\nu > 0} \|\exp(-\nu L(\B{\phi}))\| \le K.
\label{eq_def_stability}
\end{equation}

The solution error is $\varepsilon_h(t,v_0) = u(t) - \Pi_h v(t, \ \cdot\ )$ where $u(t)$ is the solution of \eqref{eq3} with the initial data $\Pi_h v_0$, and $v(t,\B{r}) = v_0(\B{r}-\B{\omega} t)$. We say that the scheme has order $q$ if for each 1-periodic $v_0 \in C^{q+1}(\mathbb{R}^2)$ there holds $\|\varepsilon_h(t,v_0)\| \le (c_1 + c_2 t) h^q$ with some $c_1$ and $c_2$ depending on $v_0$. The optimal order of accuracy is the maximal value of $q$ such that this estimate holds.

The following proposition is a particular case of Theorem~3.1 in \cite{Bakhvalov2023}.
\begin{proposition}
Let the scheme \eqref{eq3} be stable. If for each multiindex $\B{m}$, $\vert\B{m}\vert=p+1$, there holds $f^{\B{m}} \in \mathrm{Im}\,L(0)$, then the optimal order of accuracy of the scheme \eqref{eq3} is $p+1$. Otherwise the optimal order of accuracy is $p$.
\label{th:1}
\end{proposition}

If $f^{\B{m}} \in \mathrm{Im}\,L(0)$, then the leading terms of the truncation error may be represented in a divergence form, which yields the $(p+1)$-th order convergence. The second statement is more difficult to see. In \cite{Bakhvalov2023}, it was proved using the spectral analysis.

\section{Accuracy analysis of SD-RT(1)}

In this section, we study the accuracy of the SD-RT(1) scheme for \eqref{eq_TE} on the right-triangular meshes defined in Section~\ref{sect:regmesh}. Throughout this section, $\omega_x, \omega_y \ge 0$, $\omega_x+\omega_y > 0$.

Our analysis is based on Proposition~\ref{th:1}. To apply it, we need to know the stability, the co-kernel of the matrix $L(0)$, and some properties of $f^{\B{m}}$.

\subsection{Stability}

The eigenvalues of 
$$
L(\B{\phi}) = \left(L_{(0,0)}^x + L_{(-1,0)}^x \exp(-i\phi_x)\right) \omega_x + \left(L_{(0,0)}^y + L_{(0,-1)}^y \exp(-i\phi_y)\right) \omega_y
$$
were studied numerically in  \cite{SD8}. It was shown that for each $\B{\omega}$ and each $\B{\phi}$ all eigenvalues have nonnegative real parts. To establish the stability we need a stronger statement.

\begin{lemma}\label{th:stab1}
For each $\omega_x, \omega_y \ge 0$ there holds \eqref{eq_def_stability} with $K = 32$. As a corollary, the scheme is stable.
\end{lemma}
\begin{proof}
Without loss, $\omega_x = \cos\xi$, $\omega_y = \sin\xi$, $0 \le \xi \le \pi/2$. We use Lapack to find the eigenvalues and eigenvectors of $L(\B{\phi})$ for all admissible $\xi, \phi_x, \phi_y$ with the step $\pi/100$. The results show that all eigenvalues have nonnegative real parts, and the condition number of the matrix of eigenvectors does not exceed 32.
\end{proof}

\subsection{Properties of $L(0)$}

The matrix $L(0)$ is defined as
$$
L(0) = \left(L_{(0,0)}^x + L_{(-1,0)}^x\right) \omega_x + \left(L_{(0,0)}^y + L_{(0,-1)}^y\right) \omega_y.
$$

General considerations give us the following information.
\begin{itemize}
\item For each $\B{\omega}$ the vector $(1,1,1,1,1,1)^T$ belongs to the kernel of $L(0)$ because the scheme is exact on a constant solution.
\item For each $\B{\omega}$ the vector $(1,1,1,1,1,1)$ belongs to the co-kernel of $L(0)$ because the scheme is conservative.
\item If \mbox{$\B{\omega} = (1,0)^T$}, then $v=y$ is a steady solution of \eqref{eq_TE}. By 1-exactness, the vector $(0,0,1,1,1,0)^T$ also belongs to the kernel of $L(0)$.
\end{itemize}

We need to refine these results.

\begin{lemma}\label{th:L0}
If $\omega_x, \omega_y > 0$, the co-kernel of $L(0)$ is the span of $(1,1,1,1,1,1)$. 
If \mbox{$\omega_x > \omega_y = 0$}, then the co-kernel of $L(0)$ is the span of $(1,1,1,1,1,1)$ and $(5,5,2,0,0,3)$.
\end{lemma}
\begin{proof}
By the direct substitution, for non-degenerate matrices $S_L$ and $S_R$ defined by
$$
S_L = \left(\begin{array}{cccccc}
3 & 6 & 6 & 3 & 0 & 3 \\
1 & 1 & 1 & 0 & 0 & 3 \\
1 & 1 & 1 & 3 & 0 & 0 \\
2 & 1 & 2 & 1 & 0 & 1 \\
5 & 5 & 5 & 3 & 0 & 3 \\
1 & 1 & 1 & 1 & 1 & 1
\end{array} \right),
\quad
S_R = \left(\begin{array}{cccccc}
1 & 0 & 0 & 0 & -1 & 1 \\
0 & 1 & 0 & 1 & -1 & 1 \\
0 & 0 & 1 & 0 & 0 & 1 \\
0 & 0 & 0 & 0 & 0 & 1 \\
0 & 0 & 0 & 1 & 0 & 1 \\
0 & 0 & 0 & 1 & -1 & 1
\end{array} \right),
$$
there holds
$$
-\frac{1}{9} S_L L(0) S_R
=
\left(\begin{array}{cccccc}
\omega_x+\omega_y & 0 & 0 & 0 & 0 & 0 \\
0 & \omega_x+\omega_y & 0 & 0 & 0 & 0 \\
0 & 0 & \omega_x+\omega_y & 0 & 0 & 0 \\
-\omega_x/3 & \omega_y/3 & -\omega_x/3 & \omega_x & 0 & 0 \\
0 & 0 & 0 & \omega_x & \omega_y & 0 \\
0 & 0 & 0 & 0 & 0 & 0
\end{array} \right).
$$
If $\omega_x, \omega_y > 0$, then the last row of $S_L$ forms basis of the co-kernel of $L(0)$. If $\omega_x > \omega_y = 0$, then $\mathrm{dim}\,\mathrm{Ker} L(0) = 2$. It is easy to see that
$$
(-1,0,-1,-3,3,0) S_L L(0) S_R = 0.
$$
Thus, $(-1,0,-1,-3,3,0) S_L = (5,5,2,0,0,3)$ belongs to the co-kernel of $L(0)$. 
\end{proof}

\subsection{Mean truncation error on the periodic cell}


For \mbox{$f \in \mathbb{C}^6$}, let $I[f]$ be the piecewise-linear function on $(0,1)^2$ with point values $I[f](\B{a}_{\xi})=f_{\xi}$ at vertices $\xi=1,\ldots,6$ assigned according to the numbering in Fig.~\ref{fig:sol_points}. Then the orthogonality of $f^{\B{m}}$ and $(1,1,1,1,1,1)$ is equivalent to  zero integral of $I[f^{\B{m}}]$ over $(0,1)^2$. Now we show that this condition holds.

\begin{lemma}\label{th:2}
For each $\B{\omega} \in \mathbb{R}^2$ and each $\B{m}$ such that $\vert \B{m} \vert =2$ there holds
$$
Y:=\int\limits_{G} I[f^{\B{m}}](\B{r}) d\B{r} = 0, \quad G = (0,1)^2.
$$
\end{lemma}
\begin{proof}
By the definition \eqref{eq_trerr} of $f^{\B{m}}$,
\begin{equation}
Y = \int\limits_{G} I\!\left[ - \left(\Pi_1 (\B{\omega} \cdot \nabla) \frac{\B{r}^{\B{m}}}{\B{m}!}\right)_{0} + \sum\limits_{\B{\zeta} \in \mathcal{S}} L_{\B{\zeta}} \left(\Pi_1 \frac{\B{r}^{\B{m}}}{\B{m}!}\right)_{\B{\zeta}}\right]\!(\B{r}) d\B{r}.
\label{eq_l4tp}
\end{equation}

Since $(\B{\omega} \cdot \nabla) \B{r}^{\B{m}}$ is a first-order polynomial, then the map $\phi \to I[(\Pi_1 \phi)_0]$ keeps it unmodified within $G$. Therefore,
$$
\int\limits_{G} I\left[\left(\Pi_1 (\B{\omega} \cdot \nabla) \frac{\B{r}^{\B{m}}}{\B{m}!}\right)_{0}\right] d\B{r} =
\int\limits_{G} (\B{\omega} \cdot \nabla) \frac{\B{r}^{\B{m}}}{\B{m}!} d\B{r}.
$$

The coefficients $L_{\B{\zeta}}$ in \eqref{eq3} are defined so that for each piecewise linear function $u$ and both triangles $e \subset G$ there holds
$$
I\Biggl[\,\sum\limits_{\B{\zeta} \in \mathcal{S}} L_{\B{\zeta}} u_{\B{\zeta}}\Biggr](\B{r}) = \mathrm{div} \B{f}_e[u](\B{r}), \quad \B{r} \in e.
$$
Taking this with $u = \Pi_1 \B{r}^{\B{m}}/\B{m}!$ to \eqref{eq_l4tp} we obtain
$$
Y =  \sum\limits_{e \subset G} \int\limits_{e} \left(-(\B{\omega}\cdot\nabla) \frac{\B{r}^{\B{m}}}{\B{m}!} + \mathrm{div} \B{f}_e\!\left[\Pi_1 \frac{\B{r}^{\B{m}}}{\B{m}!}\right]\right) d\B{r}.
$$
By the Gauss theorem,
$$
Y = \sum\limits_{\Gamma} \int\limits_{\Gamma} \left(-(\B{\omega}\cdot\B{n}) \frac{\B{r}^{\B{m}}}{\B{m}!} + \B{n}\cdot \B{f}_e\!\left[\Pi_1 \frac{\B{r}^{\B{m}}}{\B{m}!}\right]\right) dl
$$
where the sum is by edges of triangles from $G$, and the unit normal $\B{n}$ is directed outwards. On each edge, the function $\B{n} \cdot \B{f}_e[u]$ is a first-order polynomial (see Section~\ref{sect:RT}). By \eqref{eq_SD2}, it equals the upwind limit value of $u$ multiplied by $\B{\omega}\cdot\B{n}$. Thus,
$$
Y = \sum\limits_{\Gamma} \int\limits_{\Gamma} (\B{\omega}\cdot\B{n}) \left(-\frac{\B{r}^{\B{m}}}{\B{m}!} + \left(\Pi_1 \frac{\B{r}^{\B{m}}}{\B{m}!}\right)(\B{r} \mp 0\B{n})\right) dl.
$$
We take the upper sign if $\B{\omega} \cdot \B{n} > 0$ and the lower sign otherwise. The edge located within the periodic cell counts twice with opposite normal direction and yields zero in sum. Since $\B{r}^{\B{m}} \in P_2$, then the function $\B{r}^{\B{m}} - \Pi_1 \B{r}^{\B{m}}$ is 1-periodic. Thus, for each pair of opposite edges the expression in parentheses is the same, and the normals are opposite. Therefore, the sum by edges yields zero.
\end{proof}

\begin{remark}
The proof of Lemma~\ref{th:2} extends to SD-RT($p$) for each $p \in \mathbb{N}$ and to the discontinuous Galerkin method.
\end{remark}

\subsection{Main result}

\begin{theorem}\label{th:main}
Consider the SD-RT(1) scheme for \eqref{eq_TE} on the right-triangular meshes defined in Section~\ref{sect:regmesh}. If $\omega_x, \omega_y > 0$, then the optimal order of accuracy is 2. If $\omega_x > \omega_y = 0$, then the optimal order of accuracy is 1.
\end{theorem}
\begin{proof}

If $\omega_x, \omega_y > 0$, then the vector $(1,1,1,1,1,1)$ forms the basis in the co-kernel of $L(0)$. The statement of the theorem follows from  Proposition~\ref{th:1} and Lemma~\ref{th:2}.

Consider the case $\omega_x > \omega_y = 0$. 
To evaluate the numerical derivative of a mesh function at the zeroth block, the scheme uses its values at $x=0$ and $x=1$. At these points, the values of $\Pi_1 (x^2/2)$ coincide with the values of $\Pi_1 (x/2)$. So the numerical derivative of $x^2/2$ at zeroth block is equal to the numerical derivative of $x/2$, which is $1/2$ by $1$-exactness. Thus,
$$
f^{(2,0)} = -  (\Pi_1 x)_0\ \omega_x + \left(\begin{array}{c} 1/2 \\ 1/2 \\ 1/2 \\ 1/2 \\ 1/2 \\ 1/2 \end{array}\right) \omega_x  = \left(\begin{array}{c} 1/2 \\ -1/2 \\ 1/2 \\ 1/2 \\ -1/2 \\ -1/2 \end{array}\right) \omega_x.
$$
Clearly, $f^{(2,0)}$ is not orthogonal to the vector $(5,5,2,0,0,3)$, which belongs to the co-kernel of $L(0)$ by  Lemma~\ref{th:L0}. Hence, $f^{(2,0)} \not\in \mathrm{Im}\,L(0)$. By Proposition~\ref{th:1}, the optimal order of accuracy \mbox{is 1}.
\end{proof}

\begin{remark}\label{remark:2}
Up to this moment, we used the Lagrangian mapping $\Pi_h$. However, the statement of Theorem~\ref{th:main} holds for any local mapping $\breve{\Pi}_h$ (for instance, the orthogonal $L_2$-projection) that coincides with $\Pi_h$ on linear functions. Indeed, for each smooth $f$ there holds $\|(\breve{\Pi}_h - \Pi_h) f\| \le c(f) h^2$. If the scheme is second-order in the sense of $\Pi_h$, then by stability and the triangle inequality it is second-order in the sense of $\breve{\Pi}_h$, and vice versa.
\end{remark}

\subsection{The long-time simulation accuracy}

In this section we show that SD-RT(1) possesses the second order of accuracy in the long-time simulation for each $\B{\omega}$. For $\omega_x, \omega_y > 0$, this follows from Theorem~\ref{th:main}, so we need to consider the case $\omega_x > \omega_y = 0$ only.

\begin{lemma}\label{th:lemma6}
Let $\omega_x > \omega_y = 0$. Then the scheme \eqref{eq3} is 2-exact in the sense of $\tilde{\Pi}_h$ taking each 1-periodic $v \in C^2(\mathbb{R}^2)$ to the mesh function with the components
$$
(\tilde{\Pi}_h v)_{\B{\eta},\xi} =
(\Pi_h v)_{\B{\eta},\xi} + h c_{\xi} \left(\Pi_h \left(\frac{1}{2}\frac{\D v}{\D x} - \frac{\D v}{\D y} \right)\right)_{\B{\eta},\xi} - \frac{1}{6} h^2 d_{\xi} \left(\Pi_h \frac{\D^2 v}{\D x^2}\right)_{\B{\eta},\xi},
$$
where $c_{\xi} = 1$ for $\xi = 3, 4, 5$ and $c_{\xi}=0$ otherwise, and  $d_{\xi} = 1$ for $\xi = 1, 4$ and $d_{\xi}=0$ otherwise.
\end{lemma}

The proof is by direct substitution, see Appendix~\ref{sect:A}.

\begin{corollary}
Consider the SD-RT(1) scheme \eqref{eq_02}--\eqref{eq_id} on the right-triangular meshes defined in Section~\ref{sect:regmesh}. Let $\omega_x > \omega_y = 0$. Then for each 1-periodic $v_0 \in C^3(\mathbb{R}^2)$ there holds
$$
\|\varepsilon_h(t,v_0)\| \le c(v_0) (h + th^2).
$$
\end{corollary}

The proof is standard, so we give its sketch only.

\begin{proof}
Let $\tilde{u}$ be the solution of \eqref{eq_02}--\eqref{eq_SD2} with the initial data $\tilde{u}(0) = \tilde{\Pi}_h v_0$. Then
\begin{equation*}
\begin{gathered}
\|\varepsilon_h(t,v_0)\| = \|u(t) - \Pi_h v(t, \ \cdot\ )\| \le
\\
\le
\|u(t) - \tilde{u}(t)\| + \|\tilde{u}(t) - \tilde{\Pi}_h v(t, \ \cdot\ )\| + \|\tilde{\Pi}_h v(t, \ \cdot\ ) - \Pi_h(t, \ \cdot\ )\|.
\\
\le 
K \|(\tilde{\Pi}_h - \Pi_h) v_0\| + 
\|\tilde{u}(t) - \tilde{\Pi}_h v(t, \ \cdot\ )\| +
\|(\tilde{\Pi}_h - \Pi_h) v(t, \ \cdot\ )\|.
\end{gathered}
\end{equation*}
The first and the third terms at the right-hand side are $O(h)$. The second term is the solution error by a 2-exact stable scheme and thus is $O(th^2)$.
\end{proof}


\section{Numerical results}

Now we apply SD-RT($p$) with $p = 1, 2, 3$ to the Cauchy problem for the transport equation \eqref{eq_TE} with the initial data $v_0(x,y) = \sin(2\pi(x+y))$ and the transport velocitiy $\B{\omega} = (\cos\phi, \sin\phi)$. We consider the cases $\phi=0$ and $\phi=\pi/8$ and use the right-triangular meshes described in Section~\ref{sect:regmesh}. On the boundaries of the unit square, the periodic boundary conditions are set. 

For the time integration, we use the 3-rd (for $p=1,2$) or the 4-th (for $p=3$) order Runge~-- Kutta method with the CFL number 0.1. So the error of the time integration is negligible.

For a triangle with vertices $\B{r}_1$, $\B{r}_2$, $\B{r}_3$, define the solution collocation points as $(i_1 \B{r}_1 + i_2 \B{r}_2 + i_3 \B{r}_3)/p$, $i_1, i_2, i_3 \in \mathbb{N} \cup \{0\}$, $i_1 + i_2 + i_3 = p$. For a given time, we measure the maximal difference between the numerical and the exact solution at the solution collocation points.

The numerical results at $t=0.1$ are shown in Fig.~\ref{fig:results}. We see the convergence with the order $p+1$ for $\phi = \pi/8$ and with the order $p$ for $\phi = 0$. This confirms the theoretical results proved for $p=1$. 

\begin{figure}[t]
\centering
\includegraphics[width=0.8\textwidth]{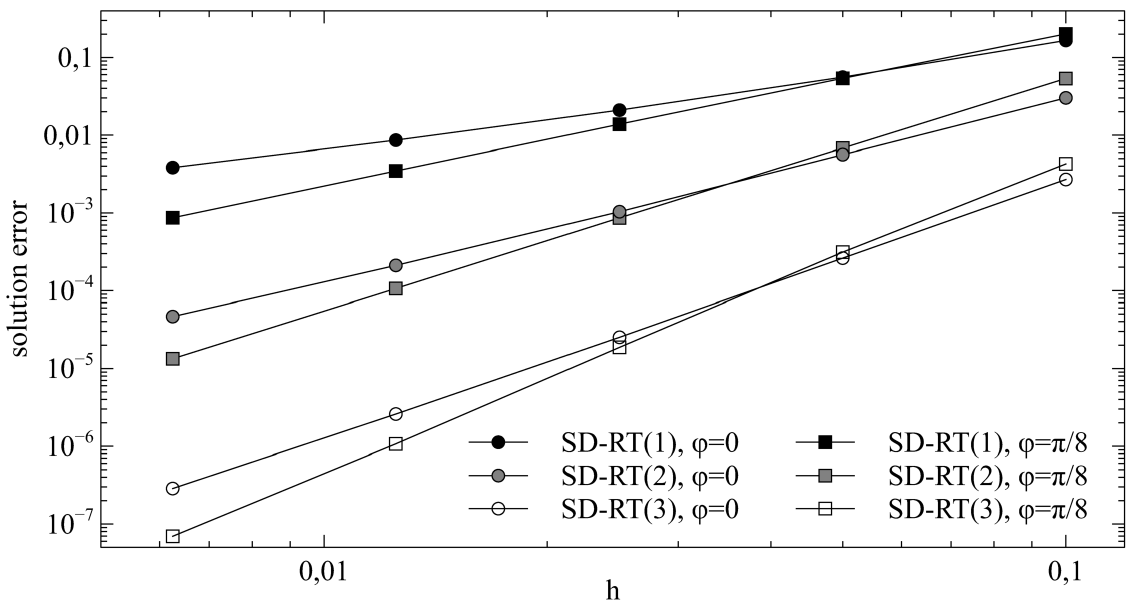}
\caption{The norm of the solution error at $t=0.1$}\label{fig:results}
\end{figure}

To study the long-time simulation accuracy, we plot the norm of the solution error as a function of time. The results are shown in Fig.~\ref{fig:results4} ($p=1$, $\phi=0$), Fig.~\ref{fig:results5} ($p=1$, $\phi=\pi/8$), Fig.~\ref{fig:results2} ($p=2$, $\phi=0$), Fig.~\ref{fig:results3} ($p=2$, $\phi=\pi/8$). Each line on each of these figures corresponds to a different $h$ ($h=0.1$, $h=0.05$, $h=0.025$, $h=0.0125$). We see that for a small time, the behavior of the solution error for $\phi=0$ differs from the case $\phi=\pi/8$. However, for $t \gg 1/h$ the distance between lines corresponding to different $h$ is identical for both cases and corresponds to the $(p+1)$-th order convergence.


\begin{figure}[t]
\centering
\includegraphics[width=0.9\textwidth]{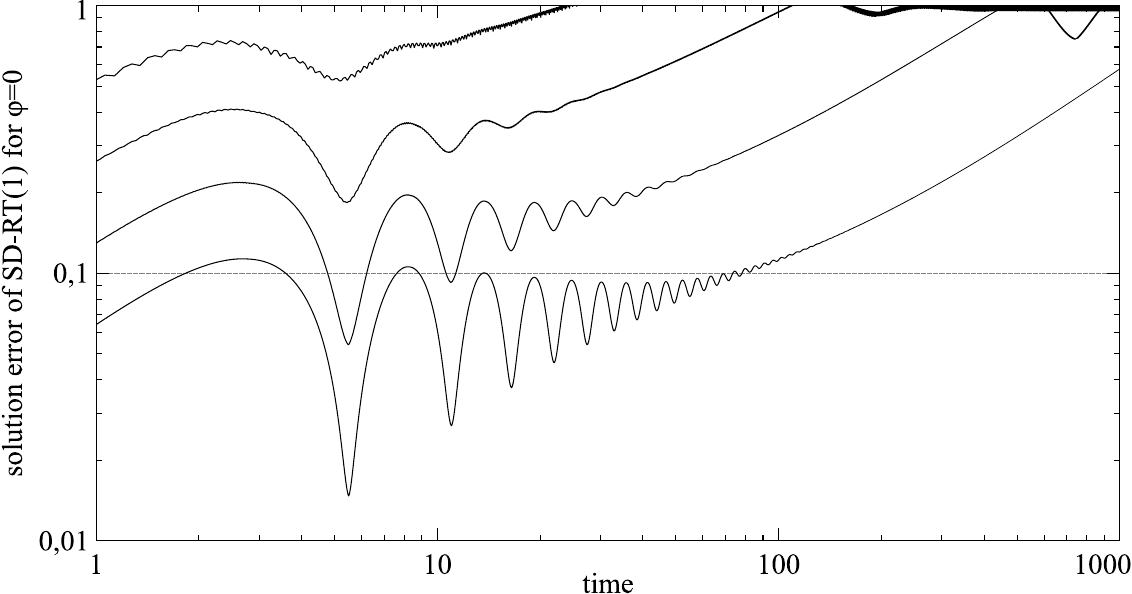}
\caption{The norm of the solution error for $p=1$, $\phi=0$}\label{fig:results4}
\end{figure}
\begin{figure}[t]
\centering
\includegraphics[width=0.9\textwidth]{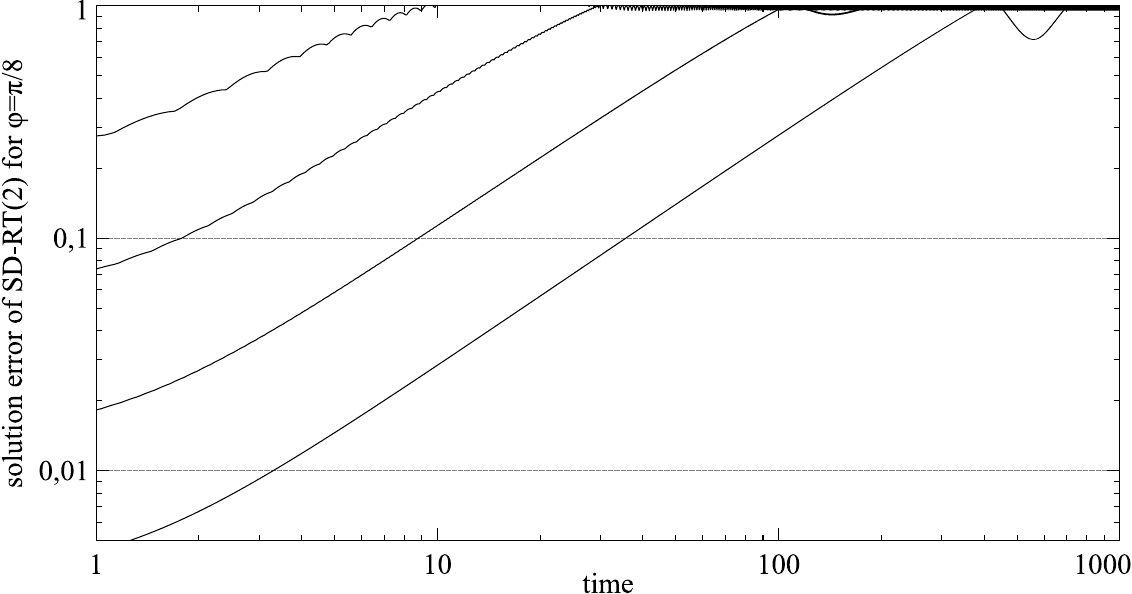}
\caption{The norm of the solution error for $p=1$, $\phi=\pi/8$}\label{fig:results5}
\end{figure}
\begin{figure}[t]
\centering
\includegraphics[width=0.9\textwidth]{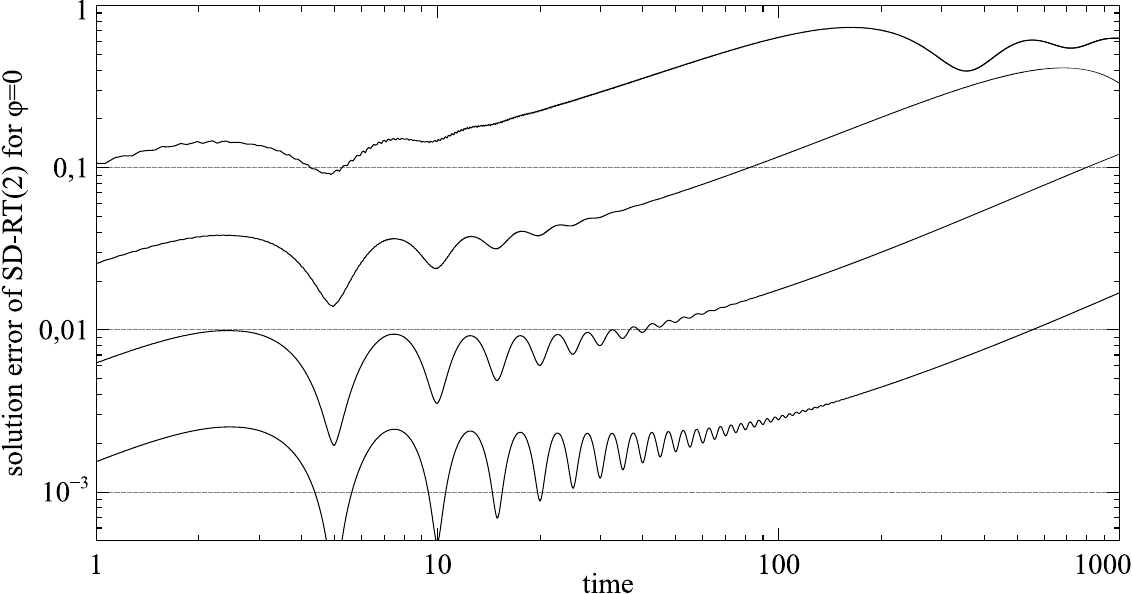}
\caption{The norm of the solution error for $p=2$, $\phi=0$}\label{fig:results2}
\end{figure}
\begin{figure}[t]
\centering
\includegraphics[width=0.9\textwidth]{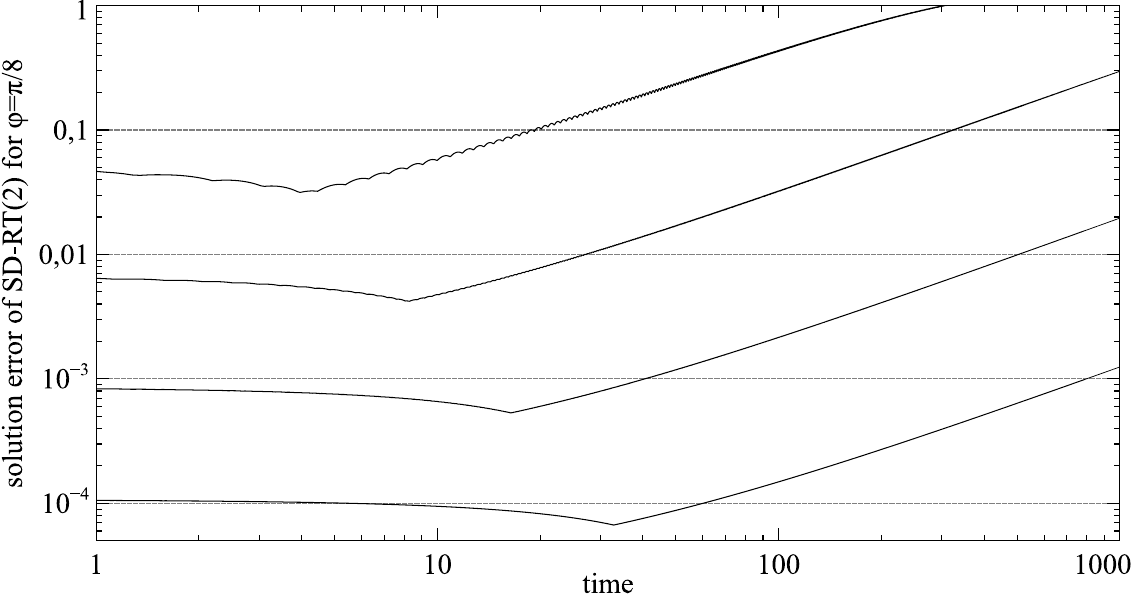}
\caption{The norm of the solution error for $p=2$, $\phi=\pi/8$}\label{fig:results3}
\end{figure}

\FloatBarrier

\appendix
\section{How to get the coefficients of $\tilde{\Pi}_h$ in Lemma~\ref{th:lemma6}}\label{sect:A}

Here we present the code for the Sagemath package that was used to find the coefficients.

First set the matrices $L_{(0,0)}$ and $L_{(-1,0)}$.

~\\
L0 = matrix([[3,1,1,0,0,0],[-3,1,-2,0,0,0],[0,1,4,0,0,0], \\~ [0,-1,-4,3,1,1],[0,2,2,-3,1,-2],[0,-4,-1,0,1,4]]) \\
Lm = matrix([[0,0,0,0,-1,-4],[0,0,0,0,2,2],[0,0,0,0,-4,-1], \\~ [0,0,0,0,0,0],[0,0,0,0,0,0],[0,0,0,0,0,0]]) \\

Now set $(\Pi_1 1)_0$, $(\Pi_1 x)_0$, $(\Pi_1 y)_0$, $(\Pi_1 xy)_0$, $(\Pi_1 x^2)_0$,  $(\Pi_1 y^2)_0$.

~\\
e = vector((1,1,1,1,1,1)); x = vector((0,1,0,0,1,1)); y = vector((0,0,1,1,1,0))\\
xy = vector((0,0,0,0,1,0)); x2 = x; y2 = y\\

Indroduce the diagonal matrices $M_x$, $M_y$, $M_{xx}$ with undetermined coefficients. 

~\\
var('c,d,b')\\
Mx = diagonal\_matrix([0,0,c,c,c,0])\\
My = diagonal\_matrix([0,0,d,d,d,0])\\
Mxx = diagonal\_matrix([b,0,0,b,0,0])\\

\noindent The form of $M_x$ and $M_y$ is defined by the 1-exactness of the scheme, and the addition of the diagonal matrix does not matter.

 Now set $(\tilde{\Pi}_1 1)_0$, $(\tilde{\Pi}_1 x)_0$, $(\tilde{\Pi}_1 y)_0$, $(\tilde{\Pi}_1 x^2)_0$, $(\tilde{\Pi}_1 xy)_0$, $(\tilde{\Pi}_1 y^2)_0$.

~\\
xp=x+Mx*e; yp=y+My*e; x2p = x2 + Mx*2*x + Mxx*2*e;\\
xyp = xy + Mx*y + My*x; y2p = y2 + My*2*y\\

Finally, write the truncation error on the quadratic polynomials in the sense of $\tilde{\Pi}_h$.

~\\
fxx = -2*xp + L0*x2p + Lm*(x2p-2*xp+e)\\
fxy = -yp + L0*xyp + Lm*(xyp-yp)\\
fyy = 0 + L0*y2p + Lm*y2p\\

The result is
$$
f^{xx} = (6b + 1, -6b - 1, -2c + 1, 6b + 1, -6b - 1, 2c - 1),
$$
$$
f^{xy} = (0, 0, -d - 1, 0, 0, d + 1), \quad f^{yy} = (0, 0, 0, 0, 0, 0).
$$
Equating $f^{xx}$ and $f^{xy}$ to zero we obtain the coefficients: $b=-1/6$, $c=1/2$, $d=-1$.

\end{document}